\documentclass[a4paper, 12pt]{article}
\usepackage{tikz}
\usepackage{comment}
\usepackage{amsfonts}
\usepackage{amsthm}
\usepackage{amsmath}
\usepackage{geometry}
\usepackage{wrapfig}

\theoremstyle{theorem}
\newtheorem{theorem}{Theorem}

\theoremstyle{definition}
\newtheorem{definition}{Definition}[section]

\theoremstyle{proposition}

\theoremstyle{lemma}
\newtheorem{lemma}{Lemma}[section]

\title {Algorithms for orthogonal partitioning into four parts}

\author {Alexey Fakhrutdinov and Oleg R. Musin}

\begin{document}

\date{}
\maketitle

\begin{abstract}  The famous pancake theorem states that for every finite set $X$ in the plane, there exist two orthogonal lines that divide $X$ into four equal parts. We propose an algorithm whose running time is linear in the number of points in $X$ and prove that this complexity is optimal. We also consider generalizations of the pancake theorem and show that orthogonal hyperplanes can be found in polynomial time.\end{abstract}  

\medskip


\section{Introduction}

 The famous {\em ham--sandwich theorem} in $\mathbb R^3$ was proposed by Steinhaus and proved by Banach; for details see \cite{BZ}.  Stone \& Tukey \cite{ST} proved the $d$--dimensional version of the theorem in a more general setting involving measures. The  discrete version of the theorem states the following (see \cite[Theorem 3.1.2]{Mat}): 

\medskip 

\noindent  {\em Let   $X_1,\dots,X_d$ be finite sets in $\mathbb R^d$. Then there exists a hyperplane $h$ that simultaneously bisects $X_1,...,X_d$.}

\medskip 

Here ``$h$ {bisects} $X_i$'' means that each of the open half--spaces defined by $h$ contains at most  $\lfloor|X_i|/2\rfloor$ points of $X_i$. %  in $\mathbb R^d$ if no more than $n/2$ points of $P$ lie in either of the open half--spaces defined by $h$.

\medskip

Let $d=1$. Then $X=X_1$ is a set in $\mathbb R$, and $h$ is the median of $X$. (The {\em median} of a set $X$ of real numbers of size $n$ is defined as the middle element of $X$ when $n$ is odd and the average of the middle two elements when $n$ is even.) It is a well-known fact that the median can be found in $O(n)$ (linear) time.

\medskip

In two dimensions a ham--sandwich cut is a line $h$ that bisects $X_1$ and $X_2$ with $n$ points in total. Edelsbrunner and Waupotitsch \cite{EW} present an algorithm that can compute $h$ in time $O(n\log{n})$. Finally,  C. -Y. Lo,  J. Matoušek, and W. L. Steiger \cite{LoMS} (see also \cite{LoS}) proved that in 2D a ham--sandwich cut can be computed in $O(n)$ time. This paper also presents algorithms for finding ham--sandwich cuts in every dimension $d > 1$. 

\medskip

One of the most famous remaining open questions in topological combinatorics is the {\em hyperplane mass equipartition problem}, originating with Gr\"unbaum  \cite{Gru60} in 1960. He asked 

\medskip

\noindent ``{\em Is it possible to find  a partition of one finite measure by  $d$ affine hyperplanes in  $\mathbb R^d$  into $2^d$ equal parts}\,?''

\medskip

\noindent 
Avis \cite{Avis} showed that for $d\ge 5$ theorems on partitioning into $2^d$ equal pieces cannot exist, since there are examples of sets for which, for any partition by $d$ hyperplanes, fewer than $2^d$ parts are obtained. It is easy to see for $d=1,2$ and known for $d=3$ by Hadwiger \cite{Had66}. This result was later rediscovered in \cite{Yao}.  Note that the case $d=4$ remains an open problem.

\medskip

Consider the case when the hyperplanes are mutually orthogonal. For a plane, this fact is well known as the (second) {\em pancake theorem}, see \cite[Sec.7]{Shashkin}: 

\medskip

\noindent 
{\em Given one finite-area, two-dimensional pancake. Then there exist two perpendicular straight lines that cut the area of the pancake into four equal pieces.}

The algorithm previously known for this problem is described in \cite{RoySt} and has $O(n \log n)$ time complexity.
In this paper, we prove the following result.

\begin{theorem}
	Given a set of points $P$  in $\mathbb{R}^2$, $|P| = n$, an orthogonal cut partitioning the set into four equal parts can be computed in optimal linear time in $n$, i.e., in $\Theta(n)$ time.
\end{theorem}
The proof of this theorem is given in Section 2.

\medskip

Recently we have found two generalizations of the pancake and ham--sandwich theorems for all dimensions \cite{Mus25}, which follows from a generalization of the results of Section 5 of \cite{Mus12} on the Borsuk–Ulam type theorem for Stiefel manifolds.

\medskip 

\noindent {\bf Theorem A.} {\em Let $X$ be a finite set in  $\mathbb R^d$. Then there exist $d$ mutually orthogonal hyperplanes such that every pair of these hyperplanes divides $X$ into four equal parts. } 

\medskip 

\noindent {\bf Theorem B.} {\em Let $X_1, \dots ,X_m$ be finite sets in  $\mathbb R^d$, $d={\delta(m)}$, where 
$$
\delta(m):=m+2^{\lfloor \log_2 m \rfloor}.
$$
Then in  $\mathbb R^{d}$ there exist two mutually orthogonal hyperplanes that simultaneously divide all $X_i$ into four equal parts.} 

\begin{theorem}
In Theorems A and B, orthogonal hyperplanes can be found in polynomial time in the total number of points.
\end{theorem}

\section{Proof of Theorem 1}
\subsection{Dual Transformation}
\begin{definition}[Function N]
Given two orthogonal cuts of $P$, whose directions are $\varphi$ and $\varphi + {\pi}/{2}$ for some $\varphi$, denote by $N\left(\varphi\right)$ the number of points in the first quadrant, which is the quadrant above both lines when $\varphi \in \left(0, \pi/2\right)$. Using this definition, one may express the desired cut as having  $N(\varphi) = \left \lfloor {n/4} \right \rfloor$. %$N(\varphi) = \left \lfloor \cfrac{n}{4} \right \rfloor$.
\end{definition}

Existence of such a cut can be proved in an elementary way: take arbitrary $\varphi$ and assume, without loss of generality, that $N\left(\varphi\right) > \left \lfloor {n}/{4} \right \rfloor$. Then the number of points in the second quadrant is less than $\left \lfloor {n/4} \right \rfloor$ so $N \left( \varphi + \cfrac{\pi}{2} \right) < \left \lfloor \cfrac{n}{4} \right \rfloor$. As during the rotation of $\varphi$, $N\left(\varphi\right)$ changes only by one when passing through some point, the existence of $\varphi^*$ with $N\left(\varphi^*\right) = \left \lfloor {n}/{4} \right \rfloor$ is guaranteed by the intermediate value theorem. Our algorithm in fact performs a search based on this simple idea.
One may use the slope $a$ of the first cutting line instead of $\varphi$ as a parameter of such a partition. We abuse notation by defining $N\left(a\right)$ to be equal to $N\left(\arctan(a)\right)$.

% We will use use point-line duality, which preserves angles.
% \begin{definition}[Dual transformation]
% Define $D$ -- duality between some points on plane and some lines on plane via the following formula:

% $$D(a, b) = \{by = ax + 1 \}$$

% \end{definition}

% Usually, such transformations have order preserving property, i.e. for given point $p$ and line $l$ the following property holds:
% $$p \: \text{above} \: l \iff D(l) \: \text{above} \: D(p)$$

% Take point $p = (a, b)$ and line $l$ with equation $y = cx + d$. Then

% $$(a, b) \mapsto y = \frac{a}{b} x + \frac{1}{b}$$
% $$y = cx + d \mapsto (\frac{c}{d}, \frac{1}{d})$$

% The above-below preserving property then restated as:
% $$b > ca + d \iff \frac{1}{d} > \frac{a}{b} \cdot \frac{c}{d} + \frac{1}{d}$$

% This equivalence holds only when $b$ and $d$ has the same sign, so we assume that all point $p \in P$ lie in upper half-plane hence intercept terms of lines in the arrangement $D(P)$ are positive.
% Also, to prevent parallel lines, we assume that there is no pair of points, lying on the same line going through origin.

\begin{definition}[Dual transformation]
Define $D$ -- duality between points in the plane and lines in the plane via the following formula:

$$D\left(a, b\right) = \{y = ax - b \}$$

\end{definition}

This duality preserves order in the following sense: for a given point $p$ and a line $l$ $p$ is above(below) $l$ if and only if $D\left(l\right)$ above(below) $D\left(p\right)$. See \cite{Edels87} for more detailed information.

\begin{definition}[$p$-levels of line arrangement]

Given a line arrangement $H$ on a plane, $|H| = n$, and an integer $1\leq p \leq n$, the $p$-level of $H$, called $L_p\left(H\right)$, is the graph of the function $f_p$ which at $x$ takes the $p-th$ smallest value among $\{h(x)| h \in H \}$. The median of the arrangement is $L_{\left\lceil \frac{n}{2} \right\rceil}\left(H\right)$ and is denoted by $\mu_H$.
See \cite{Edels87} for details and properties.
\end{definition}

% Now apply $D$ to $P$ to get arrangement $H = D(P)$ and construct its $p$-level $L_p(H)$. For a point $s \in L_p(H)$, which is also in the upper half-plane, $D(s)$ is the line partitioning $P$ into two parts with $p$, $(n-p)$ points. Note, that there is a single point in $L_p(H)$ visible with $\varphi \in [0, \pi]$ from the origin, as there is a unique cut of $P$ with direction $\varphi$ having exactly $p$ points below itself.

% \begin{definition}
% 	Given interval $T \in \mathbb{S}^1_+$, we will denote by $(T)_x$ interval on x-axis, obtained by projecting part of median(or any other p-level) visible from the origin within $T$. By its strip we mean vertical strip with basement $(T)_x$ 
% \end{definition}

\subsection{Optimal algorithm for partitioning  into four equal parts.}
%\subsection{Optimal algorithm for partitioning a set  by two orthogonal lines into four equal parts}
\begin{lemma}
	Any algorithm for orthogonal partitioning has $\Omega \left(n\right)$ time complexity.
\end{lemma}
Here, the Big-Omega notation is used to describe the {\em asymptotic lower bound} of an algorithm. 
\begin{proof}
	We prove this by contradiction by showing that an algorithm solving the orthogonal partitioning problem can be used to find the median in an unsorted array. Namely, for a given array $a_i$, $1 \le i \le n$, consider the set $P \in \mathbb{R}^2$ consisting of points $\left(a_i, e^{a_i}\right)$ and apply partitioning algorithm. Since the curve $\left(t, e^t\right)$ passes through at most three quadrants of the partition, one line intersects this curve at exactly one point, which corresponds to the median of the array. Since the optimal algorithm for finding the median has linear time complexity, any orthogonal partitioning algorithm has a linear lower bound. 
\end{proof}

% We solve the problem in the following form: given line arrangement $H = D(P)$, find two points on its median visible with $\cfrac{\pi}{2}$ from the origin, and such that the number of lines above both points is $\lfloor \cfrac{n}{4} \rfloor$.
Without loss of generality, consider the case $n \equiv 1 \mod 4$, as illustrated in the picture below.
When passing to dual problem via $D$, the orthogonality of the two cuts corresponding to the points $\left(a, b\right)$ and $\left(c, d\right)$ is expressed as $c = f\left(a\right)$ where $f\left(x\right) = - \cfrac{1}{x}$. 

\begin{figure}[h]
\centering
\includegraphics[scale=0.4]{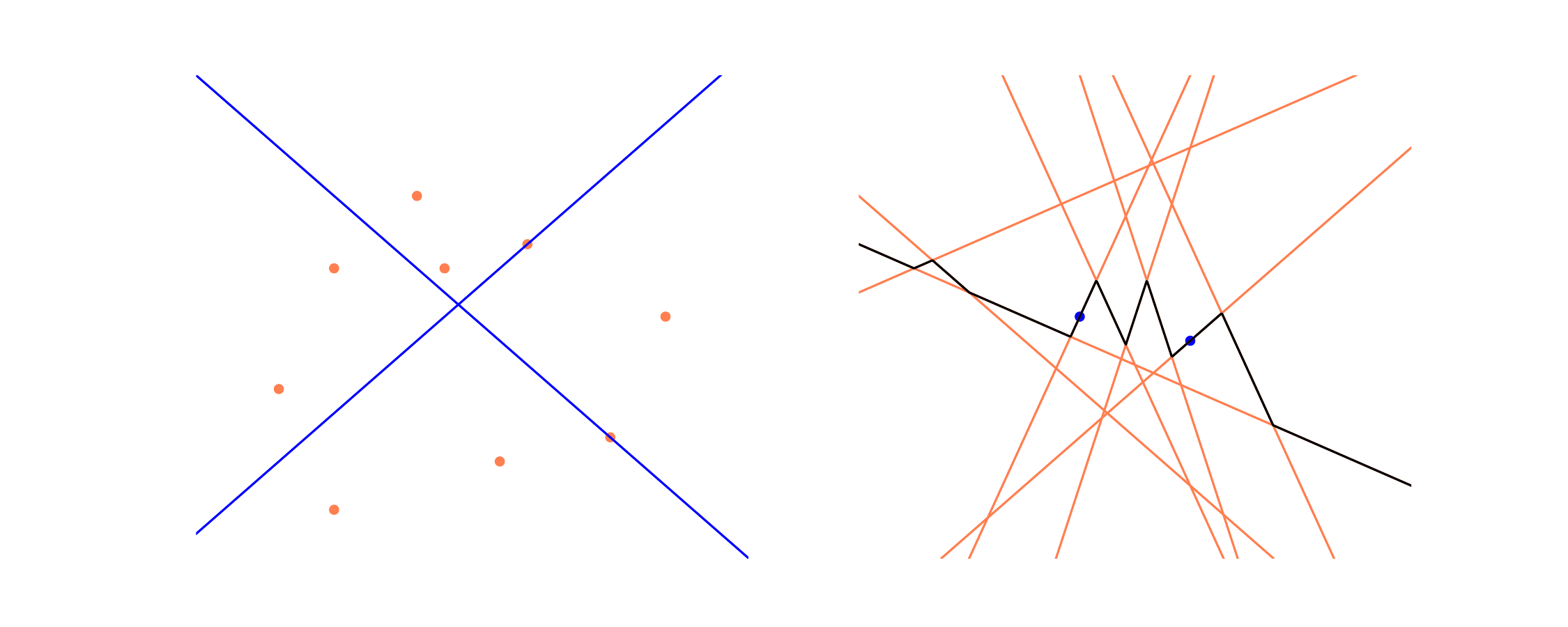}
\caption{Partitioning of the set $P$ of 9 points. The black polygonal line corresponds to the median of configuration $D\left(P\right)$. }
\end{figure}

We solve the problem in the following form: given a line arrangement $H = D(P)$, find two points on its median, $\left(a, \mu_H\left(a\right)\right)$ and $\left(b, \mu_H\left(b\right)\right)$, with $b = f\left(a\right)$, such that the number of lines above both points is $\left \lfloor {n/4} \right \rfloor$.

The algorithm discards a constant fraction of lines after each phase, until it becomes possible to solve the problem by brute force with a constant number of lines.Therefore, its time complexity is determined by that of the first phase.

At the beginning of each phase of the algorithm, we have the following data:

% \begin{itemize}
%  \item G -- set of lines remaining
%  \item $T, \: S = T + \cfrac{\pi}{2}$ -- intervals, where solution is localized
%  \item $p, \: q$ -- integers such that corresponding levels $L_p$ and $L_q$ coincide with initial median of line arrangement on $(T)_x$ and $(S)_x$ respectively.
%  \item $m$ -- current target value for function $N(\varphi)$.
% \end{itemize}

\begin{itemize}
 \item $G$ -- the set of lines remaining
 \item $T, \: S = f\left(T\right)$ -- intervals in which the solution is localized
 \item $p, \: q$ -- integers such that corresponding levels $L_p\left(G\right)$ and $L_q\left(G\right)$ coincide with the initial median of the line arrangement $\mu_H$ on $T$ and $S$, respectively.
 \item $m$ -- current target value for function $N\left(a\right)$.
\end{itemize}

% The initial values for these parameters are: $G = H$, $T = [ 0, \cfrac{\pi}{2} ], \: S = [ \cfrac{\pi}{2}, \pi ]$, $p = q = \cfrac{n + 1}{2}$, $m = \lfloor \cfrac{n}{4} \rfloor$.

The initial values for these parameters are: $$G = H, \; T = \left(0, +\infty\right), \; S = \left(-\infty, 0\right), \; p = q = (n + 1)/{2}, \; m = \left \lfloor {n}/{4} \right \rfloor. $$

The main theorem follows if we prove

% \begin{lemma}
% 	Given the data as above, it is possible in time $O(|G|)$ to compute: 
% 	\begin{itemize}
% 		\item $G^{'} \subset G$  with $|G^{'}| \le \cfrac{1}{2} |G|$
		
% 		\item $T^{'} \subset T$, $S^{'} \subset S$ with $S^{'} = T^{'} + \cfrac{\pi}{2}$
% 		\item $p^{'}$ and $q^{'}$ with $L_{p^{'}}$ and $L_{q^{'}}$ coincide with $L_p$ and $L_q$ on $(T^{'})_x$ and $(S^{'})_x$ respectively
% 		\item $m^{'}$ -- new target value for $N(\varphi)$
% 	\end{itemize}	
% \end{lemma}

\begin{lemma}
	Given the data as described above, it is possible to compute in time $O(|G|)$: 
	\begin{itemize}
		\item $G^{'} \subset G$  with $\left|G^{'}\right| \le \cfrac{1}{2} \left|G\right|$
		
		\item $T^{'} \subset T$, $S^{'} \subset S$ with $S^{'} = f\left(T^{'}\right)$
		\item $p^{'}$ and $q^{'}$ such that $L_{p^{'}}$ and $L_{q^{'}}$ coincide with $L_p$ and $L_q$ on $T^{'}$ and $S^{'}$, respectively
		\item $m^{'}$ -- the new target value for $N\left(a\right)$
	\end{itemize}	
\end{lemma}

% \begin{proof}
% We now list steps required to compute the data and after explain each step in details. 
% \begin{enumerate}
% 	\item Divide simultaneously $T$ and $S$ into sub-intervals $T_1, ..., T_C$, $S_1, ..., S_C$, $S_i = T_i + \cfrac{\pi}{2}$ with property that each sub-interval(we mean by sub-interval here $(T_i)_x$ and $(S_i)_x$) contains at most $\alpha \binom{|G|}{2}$ intersections of lines in $G$; the constant $\alpha$, on which depends $C = C(\alpha)$, will be specified later. 
% 	\item Choose one pair $(T_i, S_i)$, on whose boundary $N$ takes values greater and less than $\lfloor \cfrac{N}{4} \rfloor$, such an interval exists since $N$ takes values greater and less than $\lfloor \cfrac{N}{4} \rfloor$ on the boundary of the original interval.
% 	\item Construct trapezoids $\tau$ and $\sigma$ around $L_p$ and $L_q$ within vertical strips $(T_i)_x$ and $(S_i)_x$, such that at most half of lines meet at least one trapezoid. 
% 	\item Discard all lines missing both of the trapezoids. There are four type of discarded lines based on their position relative $\tau$ and $\sigma$. Denote by $k_{+/-, +/-}$ the number of lines go above/below $\tau$ and above/below $\sigma$.
% $p^{'} = p - (k_{-, +} + k_{-, -})$, $q^{'} = q - (k_{-, -} + k_{+, -})$ and $m^{'} = m - k_{+, +}$.
% \end{enumerate}

\begin{proof}
We now list the steps required to compute the data and then explain each step in detail. 
\begin{enumerate}
	\item Divide simultaneously $T$ and $S$ into sub-intervals $T_1, ..., T_C$, $S_1, ..., S_C$, $S_i = f\left(T_i\right)$ with property that each sub-interval contains at most $\alpha \binom{|G|}{2}$ intersections of lines in $G$; The constant $\alpha$, on which $C = C\left(\alpha\right)$ depends, will be specified later. 
	\item Choose an index $i$ such that on the boundary of $T_i$ $N(a)$ takes values both greater and less than $m$, such an interval exists since $N$ takes values greater and less than $m$ on the boundary of the original interval. In the case of $T$ is unbounded on the right we compute $N\left({\pi}/{2}\right)$ as the value of $N$ on its right boundary. This is illustrated below.

\begin{figure}[h]
\centering
\includegraphics[scale=1]{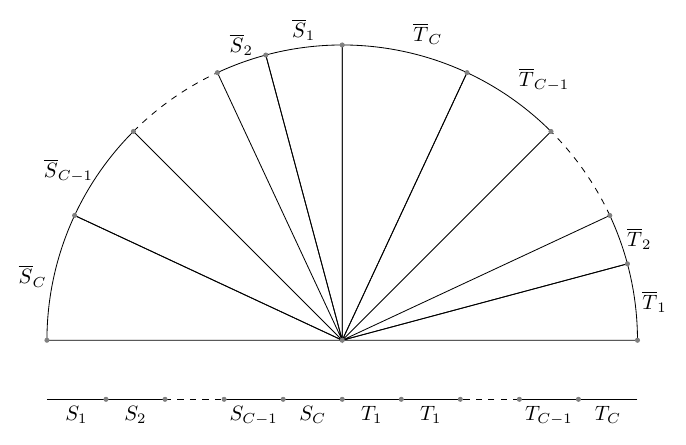}
\caption{Division of the real line into intervals and the corresponding division of the circle, given by $\overline{T}_i = \arctan \left(T_i\right)$ .}
\end{figure}

	\item Construct trapezoids $\tau$ and $\sigma$ around $L_p$ and $L_q$ within the vertical strips $T_i$ and $S_i$, such that at most half of the lines intersect at least one trapezoid. 
	\item Discard all lines that miss both trapezoids. There are four types of discarded lines based on their position relative to $\tau$ and $\sigma$. Denote by $k_{+/-, +/-}$ the number of lines go above/below $\tau$ and above/below $\sigma$.
$p^{'} = p - \left(k_{-, +} + k_{-, -}\right)$, $q^{'} = q - \left(k_{-, -} + k_{+, -}\right)$ and $m^{'} = m - k_{+, +}$.
\end{enumerate}

\begin{figure}[h]
\centering
\includegraphics[scale=0.25]{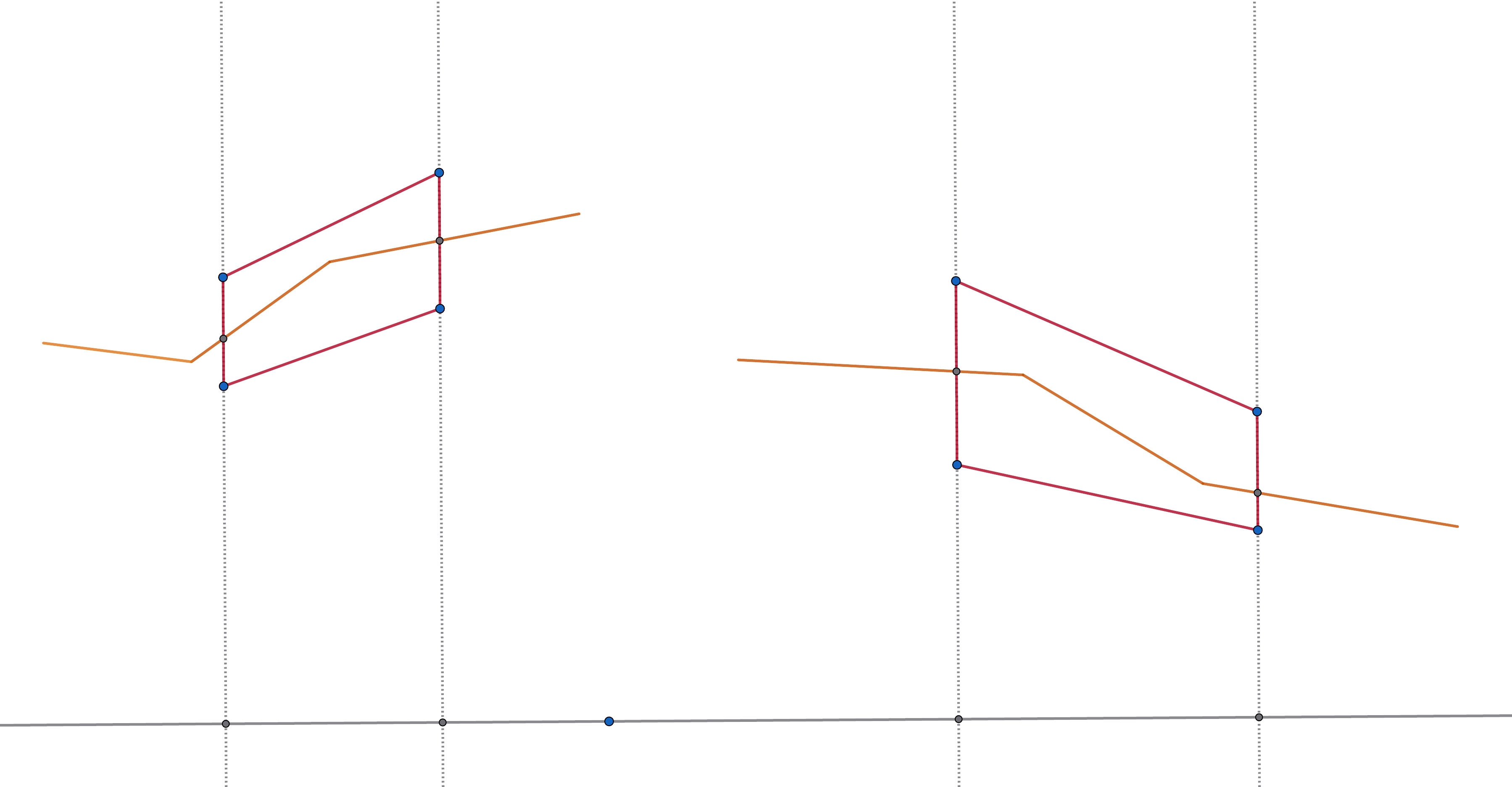}
\caption{The orange polygonal lines correspond to $L_q$ and $L_p$. Trapezoids $\sigma$ and $\tau$ are chosen to contain these levels.}
\end{figure}

We need two lemmas from \cite{LoMS}.
\begin{lemma}[Division of the interval]
Given a positive constant $\alpha < 1$, it is possible (in linear time) to divide interval $T$ into $C\left(\alpha\right)$ sub-intervals such that each $T_i$ contains at most $\alpha \binom{n}{2}$ line intersections within the $\left(T_i\right)_x$.
\end{lemma}

Step 1 of the algorithm is performed in linear time by applying this lemma independently to $T$ and $S$, and then refining the two partitions.
% One may note, that we store intervals as intervals of $\mathbb{S}^1_+$, while division algorithm in \cite{LoMS} deals with intervals of $\mathbb{R}$. To pass from $T=[ \theta_1, \theta_2] \in \mathbb{S}^1_+$ to $(T)_x$, which is a projection of part of $L_p(G)$ visible within $T$, compute two cuts corresponding to $\theta_1$ and $\theta_2$ in linear time by median-of-medians algorithm.
A more practical approach is to first divide $T$, choose one sub-interval $T_i$, then divide $S_i = f\left(T_i\right)$ and select one sub-interval $S_{ij}$. The new intervals will then be $S_{ij}$ and $f^{-1}\left(S_{ij}\right)$.

Step 2 consists simply of checking the value of $N$ at a constant number of points, each such computation is obviously done in linear time.

Next, to construct trapezoids, fix some $\varepsilon$, which will be specified later. Consider $L_p$, which is localized on some $\left(T_i\right)_x = \left[l, r\right]$. Let $D_{l/r}^{+/-}$ be the intersection of $L_{p \: +/- \: \varepsilon \left| G \right| }$ with $l/r$. The trapezoid is defined as the convex hull of these four points.

\begin{lemma}(Trapezoid properties)
	A trapezoid contains $L_p$ only if $\sqrt{{\alpha}/{2}} \leq \varepsilon$. Moreover, at most $4\varepsilon \left| G \right|$ lines intersect the trapezoid.
\end{lemma}
Since the total number of lines intersecting at least one trapezoid is bounded by $4\varepsilon \left| G \right| + 4\varepsilon \left| G \right|$, we pick $\varepsilon = {1}/{16}$, and then choose $\alpha = {1}/{128}$.
The time complexity of this step consists of computing the intersections of $L_{p \: +/- \: \varepsilon \left| G \right| }$ with $l$ and $r$, which is equivalent to finding the $p \: +/- \: \varepsilon \left| G \right| $ minimal value. This can be done in linear time using the median-of-medians algorithm.

For step 4, with the constants chosen above, we ensure that each trapezoid contains corresponding level and that at least half of the lines are discarded at each iteration.

Next, to determine the type of each line in $G$, it suffices to compute its intersections with the vertical strips $T$ and $S$ and perform a constant number of comparisons. This step is also completed in linear time.

\end{proof}

\section{Proof of Theorem 2}
\begin{proof} Here we present (non-optimal greedy) polynomial time algorithms: namely, consider all possible partitions and check which of them is valid.

\medskip

\noindent {\bf Theorem A.} Let us choose $d$ points to define the first hyperplane as their linear span. The second hyperplane is then defined by $d - 1$ points (one degree of freedom less, since it must be orthogonal to the first hyperplane), the third hyperplane by $d - 2$ points, and so on. This leads to $O\left(n^{d + \left(d-1\right) + ... + 2}\right)$ variants, each of which is tested in linear time. Hence the overall time complexity is $O\left(n^{d + \left(d-1\right) + ... + 2 + 1}\right) = O\left(n^{\frac{d\left(d+1\right)}{2}}\right)$. 

\medskip

\noindent {\bf Theorem B.}  Let $n=n_1+...n_m$, where $n_i=|X_i|$. Without loss of generality we may assume that $n_1\le n_2\le...\le n_m$. Then $n_1\le n/m$.  We choose $d$ points from $X_1$ to define the first hyperplane $h_1$ as their linear span. The second hyperplane $h_2$ is defined by some $d - 1$ points and is orthogonal to $h_1$. This gives  $O\left(n_1^{d + \left(d-1\right)}\right)$ variants.  Since each pair $(h_1,h_2)$ must be checked for all points, the overall time complexity is $O(n\,n_1^{2d-1}) \le O(n^{2d})$.  
\end{proof} 

\medskip

\noindent{\bf Remark.} A natural question is how much one can improve over this naive approach. Consider similar problem of finding a Ham-Sandwich cut in $\mathbb{R}^d$. In \cite{LoMS}, it was shown that this problem could be solved in time $O\left(n^{d-1}\right)$ while the naive approach has time complexity $O\left(n^{d+1}\right)$. It is therefore natural to expect that, for Theorem A, there exists an algorithm with time complexity $O\left(n^{\frac{d\left(d+1\right)}{2} - 2}\right)$.

\bigskip

\noindent {\bf {Acknowledgments.}}
Alexey Fakhrutdinov is supported by the ``Priority 2030'' strategic academic leadership program. The author thanks the Summer Research Programme at MIPT - LIPS-25 for the opportunity to work on this and other problems.

 \medskip

 A. A. Fakhrutdinov, MIPT Centre of Pure Mathematics.

 {\it E-mail address:} fakhrutdinov.aa@phystech.edu

 \medskip

 O. R. Musin,  University of Texas Rio Grande Valley, School of Mathematical and
 Statistical Sciences, One West University Boulevard, Brownsville, TX, 78520, USA.

 {\it E-mail address:} oleg.musin@utrgv.edu

\end{document}